\documentclass[12pt]{article}
\usepackage[cp1251]{inputenc}
\usepackage[english, russian]{babel}
\usepackage{amsfonts,amssymb,mathrsfs,amscd}
\usepackage{amsthm}
\usepackage{amsmath}
\numberwithin{equation}{section}
\title{Аналитические решения уравнений свертки на выпуклых множествах со смешанной структурой. I}

\author{ С.~Н.~Мелихов, Л.~В.~Ханина}

\date{}
\newtheorem{theorem}[]{Теорема}

\newtheorem{lemma}{Лемма}[section]

\theoremstyle{definition}
\newtheorem{definition}[]{Определение}
\newtheorem{remark}[]{Замечание}
\def\RR{\mathbb R}
\def\CC{\mathbb C}
\def\NN{\mathbb N}

\begin{document}

\maketitle
\thispagestyle{empty}

\begin{abstract}
Доказан абстрактный критерий существования линейного непрерывного правого
обратного к сюръективному оператору свертки в пространствах ростков
функций, аналитических на выпуклых подмножествах комплексной плоскости со счетным базисом окрестностей
из выпуклых областей. Он сформулирован в терминах существования специальных семейств 
субгармонических функций.

\medskip

{\bf Ключевые слова: }
уравнение свертки, пространство ростков аналитических функций,
линейный непрерывный правый обратный
\end{abstract}

%=================================================================
%=================================================================
%=================================================================

\section{Введение}\label{s1}

Пусть $Q$ -- выпуклое подмножество $\mathbb C$ с непустой внутренностью;
$A(Q)$ -- пространство всех функций, аналитических в некоторой окрестности $Q$.
В данной работе рассматриваются множества $Q$, обладающие счетным базисом окрестностей,
состоящим из выпуклых областей, а
$A(Q)$ наделяется естественной топологией счетного индуктивного предела пространств Фреше
функций, аналитических в базисных окрестностях (см. обзор В.П. Хавина \cite{KHAVIN}).
Класс таких множеств  $Q$ содержит все выпуклые области и выпуклые компакты (с непустой внутренностью)
в $\CC$. Ю.Ф. Коробейником \cite{KOR96} показано, что часть $\omega$ границы $\partial Q$ множества  $Q$ с
упомянутым свойством,
содержащаяся в  $Q$, компактна.
Поэтому (относительно) открытую часть $(\partial Q) \setminus \omega$ границы $\partial Q$
можно рассматривать как (возможное) препятствие для аналитического продолжения функций из $A(Q)$.
Зафиксируем выпуклый компакт $K$ в $\mathbb C$ и линейный непрерывный функционал $\mu$
на пространстве $A(K)$ всех ростков функций, аналитических на $K$.
Отметим, что множество $Q + K$ также имеет счетный базис окрестностей,
состоящий из выпуклых областей.
Функционал $\mu$ задает оператор свертки $T_\mu$, линейно и непрерывно отображающий
$A(Q + K)$ в $A(Q)$: \, $T_\mu (f)(z) := \mu_t (f(t + z))$.
Если $K=\{0\}$, то $T_\mu$ является дифференциальным оператором бесконечного
порядка с постоянными коэффициентами.
Хорошо известны критерии сюръективности оператора $T_\mu: A(Q+K)\to A(Q)$ для выпуклой области $Q$
(О.В. Епифанов \cite{EPIF74}, В.А. Ткаченко \cite{TKACH77}), выпуклого компакта $Q$ (О.В. Епифанов \cite{EPIF2}).
Для выпуклых множеств $Q$ со смешанной геометрической структурой различные условия сюръективности оператора
свертки $T_\mu: A(Q+K)\to A(Q)$ получены И.М. Мальцевым \cite{MAL1}, \cite{MAL2},
Ю.Ф. Коробейником \cite{KORSUR}, С.В. Знаменским и Е.А. Знаменской \cite{ZNAMS}.
Отметим, что для множеств $Q$, как в этой статье,
(ненулевой) оператор $T_\mu:A(Q+K)\to A(Q)$ сюръективен, например, в следующих ситуациях:
если $K$ совпадает с точкой или, более общим образом, если
преобразование Лапласа $\widehat\mu$ функционала $\mu$, являющееся целой функцией
экспоненциального типа, имеет индикатрису роста, равную
опорной функции $K$, и медленно убывает
(имеет вполне регулярный рост).

Для сюръективного оператора $T_\mu:A(Q+K)\to A(Q)$ возникает естественный вопрос о
наличии линейного непрерывного правого
обратного (далее ЛНПО) к оператору свертки $T_\mu : A(Q + K) \to A(Q)$
(его называют оператором решения для уравнения
свертки $T_\mu(f)=g$).
Идейно эта задача тесно связана с проблемой Л. Шварца наличия ЛНПО к
дифференциальному оператору конечного порядка с постоянными коэффициентами в пространствах
распределений и бесконечно дифференцируемых функций на открытых подмножествах $\RR^{N}$
(последняя была решена Р. Майзе, Б.А. Тейлором и Д. Фогтом \cite{MTV} в конце 80-х годов
прошлого века).
Для пространств аналитических функций упомянутая задача была вначале решена для
пространств целых функций, т.~е. для случая $Q = \CC$ и оператора $T_\mu : A(\CC) \to A(\CC)$.
Независимо друг от друга К.Д. Швердтфегер
\cite{SCHW}
и Б.А. Тейлор \cite{TAY82} доказали, что всякий ненулевой оператор $T_\mu : A(\CC) \to A(\CC)$
имеет ЛНПО.
Полное решение данной проблемы для выпуклых областей $Q$ было
получено З. Моммом \cite{UNIV} (для $K = \{0\}$), \cite{MOMM94}
(для произвольного выпуклого компакта $K$). В случае, когда $Q$ --
отрезок, критерий наличия ЛНПО к оператору $T_\mu$ был доказан
М. Лангенбрухом \cite{LANG}. Для произвольного выпуклого компакта
$Q$ соответствующий критерий установлен в статье \cite{MM}. Для многоугольных
множеств $Q$, не обязательно являющихся областями или компактами,
условия существования ЛНПО к оператору свертки $T_\mu$ установлены
Ю.Ф. Коробейником \cite{KORSB96}. В работе \cite{Melimo} был
рассмотрен случай выпуклого локально замкнутого (т. е. содержащего
относительно открытую часть своей границы) множества $Q$. Класс
таких множеств $Q$ тоже содержит все выпуклые области и выпуклые компакты,
но по своим геометрическим свойствам и вытекающим из них свойствам
рассматриваемых пространств двойственен к рассматриваемому в данной
статье. В статье А.В. Абанина и Ле Хай Хоя \cite{AKHOI} изучена проблема
наличия ЛНПО к оператору свертки $T_\mu$, действующему в пространствах аналитических функций
полиномиального роста.
Для множеств $Q$, как в данной работе, в \cite{BM} доказан абстрактный критерий
(в терминах существования специальных семейств
субгармонических функций) наличия ЛНПО к дифференциальному оператору бесконечного
порядка $T_\mu: A(Q)\to A(Q)$.
Отметим также исследования, относящиеся к многомерной
ситуации. Р. Майзе и Б.А. Тейлор \cite{MT1} показали, что любой
ненулевой оператор свертки, действующий в пространстве целых
функций многих комплексных переменных, имеет ЛНПО. Для пространств
функций, аналитических в выпуклых областях, на выпуклых компактах и на
выпуклых локально замкнутых множествах в $\CC^N$ и для операторов
свертки в них соответствующие результаты были получены,
соответственно, в \cite{ACTA}, \cite{MM95}, \cite{Melimo}.

Начиная с работ К.Д. Швердтфегера \cite{SCHW}, Б.А. Тейлора
\cite{TAY82}, Р. Майзе, З. Момма, Б.А. Тейлора \cite{MMT}
обозначилось два подхода
к решению рассматриваемой проблемы и ей подобных. Один,
восходящий к \cite{SCHW}, использует структурную
теорию пространств Фреше (различные инварианты, критерии
расщепляемости коротких точных последовательностей пространств
Фреше и, в частности, пространств степенных рядов). Второй, идущий от
\cite{TAY82}, \cite{MMT} и существенно развитый М. Лангенбрухом и З. Моммом
\cite{LANGMOMM}, связан
с применением результатов о наличии ЛНПО к
$\overline\partial$-оператору, существованием специальных семейств
(плюри)субгармони\-чес\-ких функций.
Эти подходы (применяемые зачастую одновременно)
используют
следующий общий метод. Задача о существовании ЛНПО к $T_\mu: A(Q+K)\to A(Q)$ сводится к
проблеме дополняемости подмодуля $\widehat{\mu} \cdot \mathcal A_Q$,
порожденного символом оператора свертки $T_\mu$ -- преобразованием
Лапласа $\widehat \mu$ функционала $\mu$, в пространстве $\mathcal A_{Q + K}$.
При этом $\mathcal A_Q$ (соотв. $\mathcal A_{Q+K}$) -- весовое пространство целых функций,
изоморфное посредством преобразования Лапласа сильному
сопряженному к $A(Q)$ (соотв. $A(Q+K)$). Дополняемость
подмодуля $\widehat{\mu} \cdot \mathcal A_Q$ равносильна существованию
ЛНПО к фактор-отображению $\mathcal A_{Q + K} \to \mathcal A_{Q + K}\left/\right.
(\widehat{\mu}\cdot \mathcal A_Q)$. Для решения проблемы наличия ЛНПО к
этому фактор-отображению нужна
подходящая реализация факторпространства $\mathcal A_{Q + K}
\left/\right.(\widehat{\mu\cdot} \mathcal A_Q)$ в виде весового пространства
векторнозначных последовательностей, а по сути, необходимы
"хорошие"  \, оценки снизу для символа $\widehat\mu$. Отметим, что
первый подход применялся  для пространств (исходных или связанных
с ними), имеющих несмешанную структуру: к пространствм Фреше или
счетным индуктивным пределам банаховых пространств.

В настоящей работе мы используем второй из упомянутых выше подходов. При этом
возникают трудности, связанные как со смешанной геометрической
структурой множества $Q$ и, как следствие, со смешанной
топологической структурой возникающих пространств, так
и с тем, что $Q$ может быть неограниченным.
Первая из них
частично преодолена с помощью гомологического метода,
предложенного В.П. Паламодовым \cite{PAL68}, \cite{PAL71}, и равенства ${\rm Proj}^1\mathcal X=0$,
подробно исследованного Д. Фогтом \cite{VOGT} для конкретных проективных спектров
$\mathcal X$ счетных
индуктивных пределов банаховых пространств.
%В \cite{BM} это сделано для $K=\{0\}$;
%случай произвольного выпуклого компакта $K\subset
%\mathbb C$ аналогичен.
Существенную роль в их преодолении играет описание
весового пространства $\mathcal A_Q$ целых функций, изоморфного сильному
сопряженному к $A(Q)$, как индуктивного предела весовых пространств Фреше $\mathcal A_M$,
когда $M$ пробегает семейство всех выпуклых компактных подмножеств $Q$.
Получить его позволил результат А.~Мартино
\cite{Martino} о совпадении естественных  индуктивной и проективной
топологий в пространстве ростков всех
функций, аналитических на подмножестве $\CC$.
Важную роль
играет доказанное здесь необходимое условие сюръективности $T_\mu$,
выраженное в привычных терминах медленного убывания (полной
регулярности роста) функции $\widehat\mu$ в направлениях ограниченности и негармоничности
опорной функции множества $Q$. Медленное убывание $\widehat\mu$ позволяет реализовать
описанный выше подход, получив удобное описание факторпространства
$\mathcal A_{Q + K}\left/\right.(\widehat{\mu}\cdot \mathcal A_Q)$, а
критерии дополняемости $\widehat{\mu\cdot} \mathcal A_Q$
сформулировать в терминах существования
специальных семейств субгармонических функций с равномерными
оценками сверху и локальными оценками снизу. Отметим, что соответствующее
условие $SH(A_{\widehat\mu},Q)$ существования таких семейств -- опосредованное
условие непрерывности
линейного непрерывного правого обратного к оператору, реализующему
фактор-отображение $\mathcal A_{Q + K} \to \mathcal A_{Q + K}\left/\right.
(\widehat{\mu}\cdot \mathcal A_Q)$.)

В данной статье (в \S~3) доказан критерий существования линейного непрерывного
правого обратного к оператору свертки $T_\mu: A(Q+K)\to A(Q)$ 
в терминах условия $SH(A_{\widehat{\mu}}, Q)$. (Такой критерий мы называем абстрактным.) 
В \S\,2 приводятся сведения о
множествах $Q$ рассматриваемого типа, соответствующих
пространствах и доказываются их новые свойства.
%Упомянутый критерий доказан в \S\,3.

Наличие упомянутых семейств субгармонических
функций, как и ранее для открытых и компактных выпуклых множеств $Q$,
удается выразить посредством граничного поведения
выпуклых конформных отображений в направлениях сгущения нулей $\widehat\mu$.
Указанные отображения определяются внутренностью $Q$
и замыканием $Q$, если $Q$ ограниченно, и
ограниченной части $Q$, содержащей $\omega$, если $Q$ неограниченно.
Этому будет посвящена вторая часть данного исследования, составляющая предмет следующей публикации.

%%%%%%%%%%%%%%%%%%%%%%%%%%%%%%%%%%%%%%%%%%%%%%%%%%%%%%%%%%%%%%%%%%%%%%%%%%%%%%%%%%%%%%%%%%%%%% замена для варианта 3

\section{Пространства ростков аналитических фун\-кций. Их сопряженные. Оператор свертки}\label{s2}

\subsection{Пространства ростков аналитических функций}\label{s2.1}

Далее будут рассматриваться выпуклые множества
$Q$ в $\mathbb C$,
обладающие счетным базисом окрестностей,
состоящим из выпуклых областей.
Для множества $M\subset\mathbb C$ символы
${\rm int}\,M$, $\overline M$, $\partial M$, ${\rm conv}\,M$
обозначают, соответственно, внутренность, замыкание, границу, выпуклую оболочку $M$ в
$\mathbb C$.

Приведем характеризацию множеств $Q$ рассматриваемого типа.
Полагаем $\omega:=Q\cap\partial Q$.

\begin{lemma} \label{Q0lemma} (I)
Пусть $Q$ -- выпуклое подмножество $\CC$, отличное от $\CC$,
с непустой внутренностью; $\omega\ne\emptyset$;
$Q_0 = ({\rm int}\,Q) \cup ((\partial Q)\setminus \omega)$.
Следующие утверждения равносильны:

\noindent
(i) $Q$ имеет счетный базис окрестностей, состоящий из выпуклых областей.

\noindent
(ii) Множество $\omega$ компактно, и пересечение $Q_0$
с любой прямой, опорной к $\overline Q$, замкнуто.

\noindent
(iii) Множество $\omega$ компактно
и любая прямая, опорная к $\overline Q$, не может пересекать
одновременно и $\omega$, и $(\partial Q)\setminus \omega$.

\noindent
(II) Пусть $Q$ -- выпуклое подмножество $\CC$ с пустой внутренностью.
$Q$ обладает счетным базисом окрестностей, состоящим
из выпуклых областей, тогда и только тогда,
когда $Q$ является отрезком.
\end{lemma}

\begin{proof}
Утверждение (II) следует из \cite[теорема]{KOR96}.
Равносильность (i) и (ii)
доказаны в \cite[лемма 1.1]{BM}.
При этом в \cite{BM} утверждение (ii) леммы 1.1 сформулировано неточно.
Именно, его следует сформулировать так,
как (ii) в лемме \ref{Q0lemma} здесь.
Равносильность (ii) и (iii) очевидна.
\end{proof}

В дальнейшем мы предполагаем, что $Q$ -- выпуклое
подмножество $\CC$ c непустой внутренностью, обладающее счетным базисом
окрестностей, состоящим из выпуклых областей, и
$\omega=Q\cap\partial Q\ne\emptyset$.

Положим
$$
B(t, r):=\{z\in\CC\,|\,|z-t|<r\}, \,\,
\overline B(t, r):=\{z\in\CC\,|\,|z-t|\le r\},\,\, t\in\CC, \, r>0.
$$

\begin{remark}
Базис окрестностей $Q$ образуют открытые множества
$\widetilde{Q}_n:=({\rm int}\,Q)\cup\left(\omega+B\left(0,\frac{1}{n}\right)\right)$,
\, $n\in\mathbb{N}$,
и выпуклые области
$Q_n := {\rm conv}\, \widetilde Q_n$, $n\in\mathbb{N}$.
\end{remark}

\medskip
Пусть $A(Q)$ -- пространство ростков всех функций, аналитических на $Q$,
т.~е. аналитических в некоторой открытой окрестности $Q$.
%$A(Q_n)$ -- пространство Фреше всех функций, аналитических в $Q_n$, $n\in\mathbb{N}$.
Для открытого множества $M$ в $\CC$ символ $A(M)$ обозначает
пространство всех аналитических в $M$ функций с топологией
равномерной сходимости на компактных подмножествах $M$.
С этой топологией $A(M)$ является пространством Фреше.
Поскольку $A(Q) = \bigcup\limits_{n\in\mathbb{N}} A(Q_n)$,
то введем в $A(Q)$ топологию индуктивного предела пространств $A(Q_n)$
(относительно естественных отображений вложения $A(Q_n)$ в $A(Q)$): \,
$A(Q):={\rm ind}_{n\to} A(Q_n)$.
%\subsection{}\label{s1.1}

%%%%%%%%%%%%%%%%%%%%%%%%%%%%%%%%%%%%%%%%%%%%%%%%%%%%%%%%%%%%%%%%%%%%%%%%%%%%%%%%%
%\section{Оператор свертки}\label{s2}

\medskip
Для множества $M\subset\CC$ символом $H_M$
обозначим опорную функцию $M$:\,
$H_M (z) := \sup\limits_{t\in M} {\rm Re}(zt),\, z\in \CC$.
Положим $S:= \{z\in\CC | \, |z| = 1\}$.
%Зафиксируем выпуклый компакт $K\subseteq\mathbb{C}$.

\begin{lemma} \label{convexSetBasis} Для любого выпуклого компакта $K\subset\CC$
выпуклое множество $Q + K$ обладает счетным базисом окрестностей, состоящим из
выпуклых областей.
Таким базисом является последовательность
$(Q_n + K)_{n\in\mathbb N}$, где $Q_n$ --
множества, как в замечании 1.
\end{lemma}

\begin{proof}
Отметим, что ${\rm int}(Q + K)\neq\varnothing$
и $\omega_1:=(Q + K)\cap \partial (Q + K)$ непусто.
Покажем, что множество $\omega_1$ компактно.
Пусть $z_n\in{\omega}_1$, $n\in\NN$.
Существуют $a_n \in S$ такие, что
  ${\rm Re}(z_n a_n) = H_{Q + K} (a_n) = H_Q (a_n) + H_K (a_n)$.
  Возьмем $x_n\in Q$ и $y_n \in K$, для которых $z_n = x_n + y_n$,
  $n\in\NN$.
  Поскольку
  $$
  {\rm Re}(x_n a_n)+{\rm Re}(y_n a_n)={\rm Re}(z_na_n)=
  H_Q (a_n) + H_K (a_n),
  \,\, n\in\NN,
  $$
  то
  ${\rm Re}(x_n a_n) = H_Q (a_n)$ и
  ${\rm Re}(y_n a_n) = H_K (a_n)$, \, $n\in\NN$.
  Следовательно, $x_n\in Q\cap \partial Q = \omega$.
  Поскольку $\omega$, $K$ и $S$ компактны, то найдутся
  подпоследовательности
  $(x_{n_k})_{k\in\NN}$, $(y_{n_k})_{k\in\NN}$, $(a_{n_k})_{k\in\NN}$,
  сходящиеся (в $\mathbb{C}$) соответственно к
  $x\in \omega\subset Q$, $y\in K$ и $a\in S$.

  Если $z := x + y$, то $z = \lim\limits_{k\to\infty} z_{n_k} \in Q + K$.
  Кроме того, $z\in \partial (Q + K)$. Действительно, если $z\in{\rm int}(Q+K)$, то
  найдется $\varepsilon>0$ такое, что ${\rm Re}(zb)+\varepsilon\le H_{Q+K}(b)=
  H_Q(b)+H_K(b)$ для любого $b\in S$. Поскольку ${\rm Re}(z_{n_k}a_{n_k})\to
  {\rm Re}(za)$, то это противоречит равенству
  ${\rm Re}(z_{n_k} a_{n_k}) = H_Q (a_{n_k}) + H_K (a_{n_k})$
  для больших $k$.
  Таким образом, множество $\omega_1$ компактно.

  Покажем далее, что любая опорная к
  $\overline{Q + K} = \overline{Q} + K$
  прямая не может одновременно пересекать множество
  $(\partial (Q + K))\setminus \omega_1$ и $\omega_1$.
  Предположим противное, т.~е. что найдутся
  $z\in \omega_1$ и $v\in(\partial (Q + K))\setminus \omega_1$,
  лежащие на одной опорной прямой
  $
  P:= \{t\in\mathbb C \,|\, {\rm Re} (ta) = H_{Q + K} (a)\}
  $
  для некоторого $a\in S$.
  Так как $z\in Q+K$, то существуют $x\in Q$
  и $y\in K$ такие, что $z=x+y$. При этом $x\in Q\cap\partial Q=
  \omega$.
  Поскольку $\partial(Q+K)\subseteq\partial Q + K$,
  то найдутся $u\in\partial Q$, $w\in K$ такие, что
  $v=u+w$.
  При этом отрезок $[x, u]$ лежит на опорной прямой
  $
  P_0 := \{t\in \CC \, |\, {\rm Re}(ta) = H_Q (a) \}
  $
  к $\overline{Q}$.
  Так как $P_0$, по лемме 2.1~(iii),
  не может одновременно пересекать
  $\omega$ и $(\partial Q)\backslash\omega$,
  то $u\in \omega\subseteq Q$.
  Значит, $v = u + w \in (Q + K)\cap \partial (Q + K) = \omega_1$.
  Получено противоречие.
  Следовательно, по лемме 2.1~(iii)
  множество $Q + K$ имеет счетный базис окрестностей, состоящий
  из выпуклых областей.

  Отметим, что  области
  $$
  \widetilde Q_n + K =
  ({\rm int}\,Q)\cup (\omega + B(0, 1/n)) + K =
  $$
  $$
  (({\rm int}\,Q) + K)\cup (\omega + K + B(0, 1/n)), \,
  n\in\NN,
  $$
  образуют базис окрестностей $Q + K$.
  Действительно, пусть $G$ --
  открытая окрестность $Q + K$.
  Поскольку $\omega + K$ -- компакт и
  $\omega + K \subset G$, то
  существует $n\in\NN$ такое, что
  $\omega + K  + B(0, 1/n)\subset G$.
  Кроме того, $({\rm int}\,Q ) + K \subset G$. Поэтому $\widetilde Q_n+K\subset G$.
  Поскольку $Q$ обладает базисом окрестностей,
  состоящим из выпуклых областей,
  то $(Q_n+K)_{n\in\NN}$ является базисом окрестностей $Q+K$.
\end{proof}

\medskip
Из леммы \ref{convexSetBasis} следует, что
$A(Q + K) = \bigcup\limits_{n\in\mathbb N} A(Q_n + K)$.
Введем в $A(Q + K)$ топологию индуктивного предела пространств
Фреше $A(Q_n+K)$
относительно естественных отбражений вложения $A(Q_n+K)$ в $A(Q+K)$:
$A(Q + K):= {\rm ind}_{n\to} A(Q_n + K)$.

\subsection{Сопряженное к $A(Q)$}

Пусть $(G_m)_{m\in\mathbb N}$ -- фундаментальная
последовательность компактных подмножеств ${\rm int}\,Q$. Без
ограничения общности можно считать, что все компакты $G_m$
выпуклые и $G_m \subset {\rm int }\, G_{m + 1}$ для любого $m\in\mathbb N$.

Положим
$$
G_{mn}:= {\rm conv}\left(G_m \cup \left(\omega + \overline B \left(0,
\frac{m}{n(1+m)}\right)\right)\right), m\in\NN, n\in\NN.
$$
Тогда $(G_{mn})_{m\in\NN}$
-- фундаментальная система компактных подмножеств
$Q_n$.

Положим $H_{nm} := H_{G_{nm}}, \, n,m \in \mathbb N$.
Тогда
$$
H_{nm}(z) = \max\left(H_{G_m}(z); H_{\omega}(z) + \frac{m}{n(1+m)}|z|\right),
\,z\in \mathbb C, \, n,m\in \NN.
$$

Преобразование Лапласа устанавливает линейный топологический изоморфизм
между сильным сопряженным к
$A(Q)$ и весовым пространством целых функций.
Введем его. Для $n,m \in \mathbb N$ определим весовое банахово пространство
целых функций
$$
\mathcal A_{Q, n, m} := \left\{f\in A(\mathbb C)\left|\right.\,\,
\|f\|_{n,m}:= \sup_{z\in\mathbb C} \frac{|f(z)|}{\exp(H_{nm} (z))} <+ \infty
\right\}.
$$
Ясно, что $\mathcal A_{Q,n+1,m}\subset \mathcal A_{Q,n,m}\subset \mathcal A_{Q,n,m+1}$
для любых $n, m\in\mathbb N$,
и эти вложения непрерывны.
Положим
$$
\mathcal A_{Q,n} := \bigcup\limits_{m\in\mathbb N}\mathcal A_{Q, n, m};  \,
\mathcal A_Q :=\bigcap\limits_{m\in\mathbb N}\mathcal A_{Q,n}
$$
и введем в $\mathcal A_{Q,n}$ топологию индуктивного предела пространств
$\mathcal A_{Q,n,m}$ относительно их вложений в $\mathcal A_{Q,n}$, а в
$A_Q$ -- топологию проективного предела пространств
$\mathcal A_{Q,n}$ относительно вложений $\mathcal A_Q$ в $\mathcal A_{Q,n}$:
$$
\mathcal A_{Q,n} := {\rm ind}_{m\to} \mathcal A_{Q, n, m};
\,\,
\mathcal A_Q := {\rm proj}_{\leftarrow n}{\mathcal A_{Q,n}}.
$$

Для локально выпуклого пространства $E$
символ $E'$ обозначает топологическое сопряженное к $E$
пространство. Положим $e_\lambda(z):=\exp(\lambda z)$, \,
$\lambda, z\in\CC$.

Следующая лемма доказана в \cite[лемма 2.1]{BM}.

\begin{lemma}\label{isomorfismLemma}
$(i)$ Преобразование Лапласа
$
{\cal F}(\varphi)(\lambda):=\varphi(e_{\lambda}), \,\lambda\in\mathbb C,
\,\varphi\in A(Q)',
$
%(соответственно, $\varphi \in A(Q + K)')$
является топологическим изоморфизмом сильного сопряженного
$A(Q)'_b$ к $A(Q)$
на $\mathcal A_Q$.

\noindent
$(ii)$ Указанный в (i) изоморфизм задает двойственность
между $A(Q)$ и $\mathcal A_Q$
посредством билинейной формы
$\langle g, f\rangle := {\cal F}^{-1} (f)(g)$, \,
$g\in A(Q)$,
\, $f\in \mathcal A_Q$.

%%%%%%%%%%%%%%%%%%%%%%%%%%%%%%%%%%%%%%%%%%%%%%%%%%%%%%%%%%%
%%%%%%%%%%%%%%%%%%%%%%%%%%%%%%%%%%%%%%%%%%%%%%%%%%%%%%%%%%%

\noindent
(iii) Пространство $A(Q)$
монтелевское, а значит, и рефлексивное.
\end{lemma}

\medskip
Пусть $\mathcal A_{Q, {\rm sp}}$ -- проективный спектр пространств $\mathcal A_{Q,n}$
и отображений вложения $i_{n+1}^n: \mathcal A_{Q,n+1}\to \mathcal A_{Q,n}$
(см. \cite{VOGT}).

\begin{remark}
Пространство $\mathcal A_Q$ и $\mathcal A_{Q, {\rm sp}}$ обладают следующими свойствами.

\noindent
$(i)$ $\mathcal A_Q$ полно.

\noindent $(ii)$ ${\rm Proj}^1\mathcal A_{Q, {\rm sp}}=0$.

\noindent
$(iii)$
$\mathcal A_Q$ ультраборнологично, т.\,е.
является индуктивным пределом некоторого семейства банаховых
пространств.

\noindent
$(iv)$ Всякое ограниченное подмножество $\mathcal A_Q$
равностепенно непрерывно (если $\mathcal A_Q$ отождествить с топологическим
сопряженным к $A(Q)$, как в лемме \ref{isomorfismLemma}).
\end{remark}

\begin{proof}
Утверждение (i) доказано в \cite[лемма 2.1]{BM}. По поводу доказательства
(ii) см. \cite[доказательство леммы 2.3]{BM}.

(iii): Заметим вначале, что каждое пространство
$\mathcal A_{Q_n}$, $n\in\NN$, плотно в $\mathcal A_Q$. Действительно, зафиксируем
точку $\lambda\in Q$.  Вследствие \cite[теорема 2]{POLYNOM} множество
$e_\lambda\CC[z]=\{e_\lambda f\,|\, f\in\CC[z]\}$, содержащееся в $\mathcal A_Q$,
плотно в каждом пространстве $\mathcal A_{Q,n}$. При этом
$\CC[z]$ -- множество всех многочленов над полем $\CC$. Равенство
${\rm Proj}^1\mathcal A_{Q, {\rm sp}}=0$ влечет борнологичность, а значит, и ультраборнологичность
$\mathcal A_Q$ \cite[теорема 3.4]{VOGT}.

Так как  $A(Q)$ -- бочечное локально выпуклое пространство, то (iv) вытекает из
леммы 2.3~(i) и критерия бочечности \cite[гл.~IV, 5.2]{SHEF}.
\end{proof}

\medskip
Лемма 2.3 (с учетом леммы 2.2)
дает описание и сопряженного к пространству $A(Q+K)$ для
любого выпуклого компакта $K$ в $\CC$. При этом следует
заметить, что $(G_m+K)_{m\in\NN}$ является фундаментальной
последовательностью
компактных подмножеств ${\rm int}(Q+K)$.

%%%%%%%%%%%%%%%%%%%%%%%%%%%%%%%%%%%%%%%%%%%%%%%%%%%%%%%%%%%%5

\medskip
Если $(H, F)$ -- дуальная пара (комплексных) векторных пространств относительно билинейной формы
$\langle \cdot,\cdot\rangle$, то для множества $L\subset H$ символ
$L^0$ обозначает поляру $L$ в $F$, т.~е.
$$
L^0:=\{f\in F\,|\, |\langle x,f\rangle|\le 1 \,\mbox{ для любого } \,
x\in L\}.
$$
Для векторного пространства $H$,
абсолютно выпуклого поглощающего множества $U\subset H$
через $p_U$ обозначим функционал Минковского (калибровочную функцию)
множества $U$.

\medskip
Символ $CC(Q)$ обозначает семейство всех выпуклых компактов $M\subset Q$.
Докажем теперь результат об индуктивном описании пространства $A_Q$,
играющий существенную роль при доказательстве необходимого условия
сюръективности $T_\mu$.

\begin{lemma}
\begin{itemize}
\item[(i)] Для любого ограниченного множества $B$ в $\mathcal A_Q$
существует $M\in CC(Q)$
такое, что $B$ содержится и ограниченно в $\mathcal A_M$.
\item[(ii)]
Выполняется алгебраическое и топологическое равенство
$$
\mathcal A_Q = {\rm ind}_{M\in CC(Q)} \mathcal A_M,
$$
т.~е. $A_Q=\bigcup\limits_{M\in CC(Q)} \mathcal A_M$
и $\mathcal A_Q$ является индуктивным пределом пространств Фреше
$\mathcal A_M$, $M\in CC(Q)$, относительно отображений вложения
$\mathcal A_M$ в $\mathcal A_Q$.
\end{itemize}
\end{lemma}

\begin{proof}
(i): Будем рассматривать дуальную пару $(A(Q), \mathcal A_Q)$
относительно билинейной формы $\langle\cdot,\cdot\rangle$,
как в лемме 2.3~(ii). Пусть $B$ -- ограниченное подмножество $\mathcal A_Q$. По замечанию 2
$B$ равностепенно непрерывно в $A_Q$.
Значит, существует замкнутая абсолютно выпуклая окрестность
нуля $U$ в $A(Q)$ такая, что
$B\subset U^0$.
По \cite[утверждение 1.9]{Martino}
индуктивная топология, которую мы ввели в $A(Q)$,
совпадает с другой естественной топологией -- проективной
топологией $pr$ в
$(A(Q), pr) := {\rm proj}_{M\in CC(Q)} A(M)$.
Поэтому
существуют $M\in CC(Q)$, абсолютно выпуклая окрестность нуля $V$ в
$A(M)$ такие, что
$V\cap A(Q)\subset U$. Значит, на $A(Q)$ выполняется неравенство
$p_U\le p_{V\cap A(Q)}$.
Поскольку $|\langle g,f\rangle|\le p_U(g)$
для любых $g\in A(Q)$, $f\in U^0$, то
$$
|\langle g,f\rangle|\le p_{V\cap A(Q)}(g)=p_V(g), \,\, g\in A(Q),
f\in U^0.
$$
Если взять $g:=e_\lambda$, то из последнего неравенства получим, что
для любого $f\in U^0$
\begin{equation}
|f(\lambda)|\le p_V(e_\lambda),\, \lambda\in\mathbb C.
\end{equation}
Пусть $A^\infty(M+B(0,1/n))$, $n\in\mathbb N$, -- банахово пространство
всех функций, аналитических и ограниченных в $M+B(0,1/n)$
с обычной $\sup$-нормой. Тогда $A(M)={\rm ind}_{n\to}A^\infty(M+B(0,1/n))$.
Поэтому для любого $n\in\mathbb N$ найдется постоянная $C_n$,
для которой выполняется неравенство
$p_V(g)\le C_n\sup\limits_{z\in M+B(0,1/n)}|g(z)|$,
\, $g\in A^\infty(M+B(0,1/n))$.
Значит, для любого $\lambda\in\mathbb C$
\begin{equation}
p_V(e_\lambda)\le C_n\exp(H_M(\lambda)+|\lambda|/n).
\end{equation}
Из (2.1), (2.2) следует, что для любого $n\in\mathbb N$
$$
\sup\limits_{f\in B}\sup\limits_{\lambda\in\mathbb C}
\frac{|f(\lambda)|}{\exp(H_M(\lambda)+|\lambda|/n)}<+\infty.
$$
Следовательно, множество $B$ содержится и ограниченно в $\mathcal A_M$.

(ii): Так как  любой компакт $M\in CC(Q)$ содержится
в каждой области $Q_n$, $n\in\mathbb N$,
то $\mathcal A_M$ непрерывно вложено
в $\mathcal A_Q$ для любого $M\in CC(Q)$. Из утверждения (i)
вытекает, что, наоборот, $\mathcal A_Q$ вложено в
$\bigcup\limits_{M\in CC(Q)} \mathcal A_M$, а отображение вложения
$\mathcal A_Q$ в ${\rm ind}_{M\in CC(Q)} \mathcal A_M$ переводит
всякое ограниченное в $\mathcal A_Q$ множество
в ограниченное в ${\rm ind}_{M\in CC(Q)} \mathcal A_M$
множество. Так как $\mathcal A_Q$ борнологично (замечание 2), то
это вложение непрерывно \cite[гл.~2, утверждение 8.3]{SHEF}. Значит, пространства
$\mathcal A_Q$ и ${\rm ind}_{M\in CC(Q)} \mathcal A_M$ совпадают алгебраически
и топологически.

\end{proof}

%\medskip
\subsection{Операторы свертки}

Пусть $K$ -- выпуклый компакт в $\CC$ и
$\mu\in A(K)'$.
Согласно \cite{KOR} оператор свертки
$$
T_{\mu} (f)(z) = {\mu}_t (f(t + z))
$$
линейно и непрерывно отображает каждое пространство
$A(Q_n + K), \, n\in\mathbb N$, в $A(Q_n)$.
Значит, $T_{\mu}$ -- линейное непрерывное отображение $A(Q + K)$
в $A(Q)$.

%%%%%%%%%%%%%%%%%%%%%%%%%%%%%%%%%%%%%%%%%%%%%%%%%%%%%%%%%%%%%%%%%%%%%%%%

\medskip
Пусть $\widehat\mu=\mathcal F(\mu)$ --
преобразование Лапласа $\mu$.

\begin{remark}
Если $A(Q + K)'$ и $A(Q)'$ отождествить с
$\mathcal A_{Q + K}$ и $\mathcal A_Q$ соответственно (см. лемму 2.3), то
сопряженным к оператору свертки
$T_{\mu} : A(Q + K)\to A(Q)$
является оператор умножения
$M_{\widehat \mu} : \mathcal A_Q \to \mathcal A_{Q + K}$,
$f \mapsto \widehat{\mu} f$.
\end{remark}
%\medskip
%\noindent {\bf Замечание 2.}

Приведем общий результат о существовании ЛНПО к оператору $T_{\mu}$
(см. \cite[лемма 4.1]{BM}; при доказательстве
импликации $(iii)\Rightarrow (ii)$ существенно используется ультраборнологичность
пространства $\mathcal A_{Q+K}$, имеющая место по замечанию 2).

\begin{lemma}
Следующие утверждения равносильны:

\noindent
$(i)$ $T_{\mu} : A(Q + K)\to A(Q)$ имеет ЛНПО.

\noindent
$(ii)$ Оператор $M_{\hat \mu}: \mathcal A_Q\to \mathcal A_{Q + K}, \,
f\mapsto \widehat\mu f$,
имеет линейный непрерывный левый обратный.

\noindent $(iii)$ Подпространство $\hat\mu \cdot \mathcal A_Q$ замкнуто в $\mathcal A_{Q+K}$ и
фактор-отображение $q:\mathcal A_{Q + K} \to \mathcal A_{Q + K} / (\widehat\mu \cdot
\mathcal A_Q)$ имеет ЛНПО.
\end{lemma}

%%%%%%%%%%%%%%%%%%%%%%%%%%%%%%%%%%%%%%%%%%%%%%%%%%%%%%%%%%%%%%%%%%%%%%%%%%%%%%%

\section{Условия существования линейного непрерывного
правого обратного к оператору све\-ртки}

%%%%%%%%%%%%%%%%%%%%%%%%%%%%%%%%%%%%%%%%%%%%%%%%%%%%%%%%%%%%%%%%%%%%%%

\subsection{Условия сюръективности оператора свертки}

Далее $K$ -- выпуклое компактное подмножество $\CC$, $\mu\in A(K)'\backslash\{0\}$.

Покажем, что для
сюръективного оператора $T_\mu$, как и в несмешанных случаях
(когда $Q$ открыто или компактно),
справедлив соответствующий аналог теоремы деления.

\medskip
\begin{lemma}\label{muLemma1}
Пусть оператор $T_{\mu}: A(Q + K) \to A(Q)$ сюрьективен. Тогда
$$
\widehat{\mu}\cdot  \mathcal A_Q =(\widehat{\mu}\cdot A(\CC))\cap \mathcal A_{Q+K}.
$$
\end{lemma}

\begin{proof}
Без ограничения общности $0\in Q+K$.
По теореме об открытом отображении \cite[теорема 6.7.1]{Edwards}
$T_\mu : A(Q + K)\to A(Q)$ открыто.
Поэтому подпространство ${\rm Im} \,M_{\widehat \mu} = \widehat\mu \cdot \mathcal A_Q$
замкнуто в $\mathcal A_{Q + K}\simeq A(Q + K)_b'$.
Кроме того, подпространство
$(\widehat\mu \cdot A(\CC)) \cap \mathcal A_{Q + K}$ также замкнуто в $\mathcal A_{Q + K}$.
Пусть $CC_0(Q + K)$ -- семейство всех выпуклых компактных подмножеств
$Q + K$, содержащих 0. По лемме 2.4
$\mathcal A_{Q + K} = {\rm ind}_{M\in CC_o(Q + K)} \mathcal A_M$.
По \cite[теорема 4.4]{KRTER} $\widehat{\mu}\cdot\CC [z]$
плотно в $(\widehat\mu\cdot A(\CC)) \cap \mathcal A_M$ для любого
$M\in CC_0(Q + K)$ ($\CC [z]$ -- множество всех многочленов
над полем $\CC$).
Значит, $\widehat \mu \cdot \mathcal A_Q \supset \widehat\mu \cdot \CC[z]$
плотно в $(\widehat \mu \cdot A(\CC))\cap \mathcal A_{Q + K}$.
Таким образом, $\widehat \mu \cdot \mathcal A_Q = (\widehat \mu \cdot A(\CC)) \cap
\mathcal A_{Q + K}$.
\end{proof}
%%%%%%%%%%%%%%%%%%%%%%%%%%%%%%%%%%%%%%%%%%%%%%%%%%%%%%%%%%%%%%%%%%%%%%%%%%%%%%%%%%%%%%%%%%%%

Ниже будем использовать следующие
множества опорных направлений, соответствующих опорным точкам из $\gamma\subset \partial Q$:
$$
S_\gamma:=\{a\in S\,|\,{\rm Re}(wa)=H_Q(a) \mbox{ для некоторого }
w\in\gamma\}.
$$

%$$
%F_\sigma:=\{w\in \partial Q\,|\,{\rm Re}(wa)=H_Q(a) \mbox{ для
	%некоторого } a\in\sigma\}.
%$$

Докажем некоторые геометрические свойства $Q$.

\begin{remark}	
	
\noindent
$(i)$ Если $\gamma\subset\partial Q$ компактно, множество
$S_\gamma$ тоже компактно. В частности, $S_\omega$ компактно.
	
\noindent
$(ii)$ Для любой (относительной) окрестности $\gamma$ множества $\omega$ в $\partial Q$
	множество $S_\gamma$ является (относительной) окрестностью $S_\omega$ в $S$.
%существует окрестность $U$ множества $S_\omega$ в $S$,
%	для которой $H_Q = H_\gamma$ на $U$.
	
	%Отсюда следует, что на $U$ функция $H_Q$ ограничена.
	
\noindent	
$(iii)$ Для любого компакта $\sigma\subset S\backslash S_\omega$,
любого $T\in CC(Q)$ существует
	выпуклый компакт $T_0$ в ${\rm int}\,Q$ такой, что
$H_T\le H_{T_0}$ на $\sigma$.

\noindent
$(iv)$ Для любого компакта $\sigma\subset S\backslash S_\omega$
существует $s\in\mathbb N$, для которого $H_\omega < H_{G_s}$ на $\sigma$.

\end{remark}
\begin{proof}
Утверждение $(i)$ очевидно.

$(ii)$: Предположим противное. Тогда найдутся последовательности $a_n\in S$
и $z_n \in \partial Q$, $n\in \NN$, такие, что
${\rm dist}(a_n, S_{\omega}) \to 0$ и ${\rm Re} (z_n a_n) > H_{\gamma} (a_n)$.
Без ограничения общности $a_n \to a$, где $a\in S_{\omega}$.
Возьмем точку $z_0 \in \omega$, для которой ${\rm Re} (z_0 a) = H_Q (a) = H_{\omega} (a)$.
Тогда $H_Q(a)=H_\gamma(a)$.
Заметим, что $z_n \notin \gamma $, $n\in\NN$, а значит, найдется $\delta > 0$
такое, что ${\rm dist}(z_n, \omega) \geq \delta$ для любого $n\in \NN$.
Следовательно, существует ограниченная последовательность точек
$\widetilde z_n \in \overline{Q},\, n\in \NN$, точек, расположенных на отрезках $[z_n, z_0]$
и вне множества $\omega + B(0, \delta)$.
Пусть $\widetilde z_n = (1-{\beta}_n) z_n + \beta_n z_0$, $\beta_n \in [0, 1]$.
Без ограничения общности $\beta_n \to \beta \in [0,1]$ и
$\widetilde z_n \to  v \in\overline Q$, причем $v\notin Q$.
Так как
${\rm Re} (\widetilde z_n a_n) = (1 - \beta_n) {\rm Re (z_n a_n)} + \beta_n {Re (z_0 a_n)} \ge
(1 - \beta_n) H_\gamma (a_n) + \beta_n {\rm Re} (z_0 a_n)$,
то, переходя к пределу при $n\to\infty$, получим:
${\rm Re} (va) \ge (1 - \beta) H_{\gamma} (a) + \beta H_Q (a) =
(1 - \beta) H_Q  (a) + \beta H_Q (a) = H_Q (a).$
Отсюда следует, что $v$ лежит на опорной прямой
$\{ z\in \CC \,|\, {\rm Re} (za) = H_Q (a)\}$ к $\overline Q$,
содержащей также точку $z_0\in\omega$.
Получено противоречие с леммой 2.1 (iii).

$(iii)$: Без ограничения общности $0\in{\rm int}\,Q$.
Возьмем $R>0$ такое большое, что $T$ содержится
в круге $B(0,R):=\{z\in\mathbb C\,|\, |z|<R\}$. Положим $Q_R:=Q\cap B(0,R)$.
Тогда $H_T< H_{Q_R}$ на $\sigma$. Действительно, возьмем $a\in\sigma$
%предположим, что
%найдется $a\in B$, для которого $H_T(a)=H_{Q_R}(a)$.
и точку $z_0\in T$ такую, что $H_T(a)={\rm Re}(z_0 a)$.
Опорная к $T$ прямая $p=\{z\in\CC\,|\, {\rm Re}(z a)=H_T(a)\}$ пересекает ${\rm int}\,Q$
(иначе $a\in S_\omega$). Следовательно, прямая
$p$ пересекает и ${\rm int}\,Q_{Q_R}$, а значит,
$H_T(a)={\rm Re}(z_0 a)< H_{Q_R}(a)$.
Поскольку $0\in {\rm int}\,Q_R$, то существует $\alpha\in(0,1)$
такое, что $H_T\le\alpha H_{Q_R}$. Следовательно, в качестве компакта $T_0$
можно взять $T_0:=\alpha\overline{Q_R}$.

Утверждение (iv) следует из (iii), поскольку
${\rm conv}\,\omega$ -- выпуклый компакт в $Q$.
\end{proof}

%%%%%%%%%%%%%%%%%%%%%%%%%%%%%%%%%%%%%%%%%%%%%%%%%%%%%%%%%%%%%%%%%%%%%%%%%%%%%%%%%%%%%%%%%%%%%%%%%%%%%%%%%%%%%%%%%%%%

\begin{remark} Пусть $Q$ неограниченно. Далее нам понадобится
	разделение $Q$ на две части: одна из них ограниченна и содержит
	$\omega$, а другая неограниченна. Для этого введем следующую терминологию.
	Вещественную прямую $l$ в $\CC$ назовем разделяющей $Q$, если
	$\omega$ содержится в одной открытой полуплоскости $\Pi(l)$
	с границей $l$ и пересечение этой полуплоскости с $Q$
	ограниченно. % и содержит $0$.
Такие прямые существуют.
	Возьмем точку $z_0\in\omega$.
	Существует $a\in S$ такое, что $H_Q(a)={\rm Re}(z_0 a)$ и
	$$
	Q\subset\{z\in\CC \,|\, {\rm Re}(z a)\le H_Q(a)\}.
	$$
	Так как $\omega$ компактно, то найдется $\gamma\in\RR$,
	для которого множество
	$$
	\Omega:=\{z\in Q \,|\, \gamma<{\rm Re}(z a)\le H_Q(a)\}.
	$$
	содержит $\omega$ и $0$.
	Выпуклое множество $\Omega$ ограниченно. Действительно,
	предположим, что найдутся $z_n\in\Omega$, $n\in\NN$, такие, что
	$|z_n|\to\infty$ и ${\rm dist}(z_n,\omega)\ge 1$. Пусть точки
	$w_n\in[z_0,z_n]$ выбраны так, что ${\rm dist}(w_n,\omega)=1$,
	$n\in\NN$. Тогда
	$w_n\in Q$.
	Существует подпоследовательность $(w_{n_k})_{k\in\NN}$,
	сходящаяся к некоторой точке $w_0$.
	Если $\alpha_n$ -- наименьший угол между опорной (к
	$\overline Q$) прямой $l_0:=\{z\in\CC\,|\,{\rm
		Re}(z a)=H_Q(a)\}$ и отрезком $[z_0,z_n]$, то $\alpha_n\to 0$.
	Поэтому $w_0\in l_0$ и ${\rm
		dist}(w_0,\omega)\ge 1$. Кроме того, $w_0\in\overline Q$.
	Следовательно, $w_0\in (\partial
	Q)\backslash\omega$. Этого не может быть вследствие леммы 2.1~(iii).
	
	Для всякой разделяющей $Q$ прямой $l$ положим
	$Q(l):=Q\cap\Pi(l)$.
	\end{remark}

%%%%%%%%%%%%%%%%%%%%%%%%%%%%%%%%%%%%%%%%%%%%%%%%%%%%%%%%%%%%%%%%%%%%%%%%%%%
%%%%%%%%%%%%%%%%%%%%%%%%%%%%%%%%%%%%%%%%%%%%%%%%%%%%%%%%%%%%%%%%%%%%%%%%%%%

Далее мы будем использовать субгармонические мажоранты, применяемые
в теоремах деления. Конструкции подобного рода стали широко использоваться после работ
Л.~Эренпрайса (см., например, \cite{EHRDIV},
\cite{EHREN}). При этом будут учитываться
модификации, осуществленные З.~Моммом \cite{DIVISION}.

Для ограниченного множества $M\subset\CC$, $b\in\mathbb C$, $r>0$ через $h(M,b,r)(z)$
обозначим наибольшую субгармоническую функцию,
равную  опорной функции $H_M(z)$ множества $M$ на $|z - b|\geqslant r$. Функция
$h(M,b,r)(z)$ гармоническая в круге $|z-b|<r$ и непрерывная в $|z-b|\le r$.
\begin{remark} (I) Пусть $M$ -- ограниченное подмножество $\CC$.

\noindent
(a) Опорная функция
$H_M$ удовлетворяет условию Липшица. Именно, для любых $z, w\in\CC$
\begin{equation}
|H_M(z)-H_M(w)|\le C|z-w|,
\end{equation}
где $C:=\max\limits_{|t|\le 1}|H_M(t)|<+\infty$.

\medskip
Отметим ниже некоторые простые свойства функций $h(M,b,r)$.

\noindent
(b) Функция $h(M,b,r)(b)$ непрерывна по $b\in\CC$.

\noindent
(c) Для $b\in\CC$, $r>0$ положим
$$
\Delta (M,b,r) := \max\limits_{|z - b|
\leqslant r} (h(M,b,r)(z) - H_M (z));
$$
$\Delta (M,b,0) :=0$.
\medskip
Функция $\Delta (M,b,r) $ обладает следующими свойствами:
\begin{itemize}
\item[(i)] $\Delta (M,b,r)$ %и $h(M, b, r)(b)$
положительно однородная (степени 1)
относительно $(b,r)$, т.~е.
$\Delta (M,\alpha b,\alpha r)=\alpha \Delta (M,b,r)$
%$h(M, \alpha b, \alpha r)(b)=\alpha h(M, b, r)$
для любых $b\in\CC$, $r\ge 0$, $\alpha\ge 0$.
\item[(ii)] Семейство функций $f_b(r):=
\Delta(M,b,r)$, $b\in S$, равностепенно непрерывно
в точке $r=0$.

\end{itemize}

\noindent (II) Если $Q$ ограниченно, то в точке $r=0$ равностепенно непрерывно
семейство функций $f_{b,M}(r):=
\Delta(M,b,r)$, $b\in S$, $M\in CC(Q)$.
\begin{proof}
Утверждение (b) следует из теоремы о среднем для гармонических функций и
неравенства (3.1).

Свойство (c)\,(i) вытекает из положительной однородности
    (степени 1) функции $H_M$, а свойство (c)\,(ii) следует из неравенства (3.1).

Утверждение (II)следует из неравенства (3.1) и того, что для любого
$M\in CC(Q)$ выполняется неравенство $\max\limits_{|t|\le 1}|H_M(t)|\le
\max\limits_{|t|\le 1}|H_Q(t)|<+\infty$.
\end{proof}

\end{remark}

\medskip
Для множества $M\subset\CC$
символом $E_M$ обозначим множество всех
$b\in S$ таких, что
$H_M(b) < +\infty$, а $N_M$ -- множество всех $b\in S$,
не имеющих открытой окрестности, в которой функция
$H_M$ гармонична. Положим
$$
R_M:=E_M\cap N_M.
$$
Отметим, что $N_M$ замкнуто.

\medskip
\begin{lemma}\label{CC0Lemma}
Пусть множество $Q$ ограниченно.
Тогда

\noindent
(i) $\inf\limits_{b\in R_Q} (h(Q,b,r)(b) - H_Q(b)) > 0$
для любого $r>0$.

\noindent
(ii) Для любого $r>0$ найдется компакт $M\in CC(Q)$ такой, что
$$
\inf\limits_{b\in R_Q}(h(M,b,r)(b) - H_Q(b)) > 0.
$$

\end{lemma}

\begin{proof}
(i): Предположим, что для некоторого $r>0$
$$
\inf\limits_{b\in R_Q} (h(Q,b,r)(b) - H_Q(b))=0.
$$
Тогда найдется последовательность $b_j\in R_Q$, $j\in\NN$,
для которой
$$
h(Q,b_j,r)(b_j) - H_Q(b_j)\to 0.
$$ 
Без ограничения общности
существут $b\in S$ такое, что $b_j\to b$. При этом $b\in R_Q$.
Вследствие непрерывности функции $H_Q$ и
функции $h(Q,a,r)(a)$ по $a\in S$ (см. замечание 4)
выполняется равенство $h(Q,b,r)(b)=H_Q(b)$. По принципу максимума для
субгармонических функций $H_Q=h(Q,b,r)$ в круге
$|z-b|\le r$. Это противоречит тому, что $b\in R_Q$.

(ii): Зафиксируем $r>0$. По (i)
$$
\alpha:=\inf\limits_{b\in R_Q} (h(Q,b,r)(b) - H_Q(b)) > 0.
$$
Без ограничения общности можно считать, что $0\in {\rm int}\,Q$.
Выберем $\varepsilon\in (0,1)$ такое, что
$\varepsilon<\alpha/(2\max\limits_{|z|\le r+1}|H_Q(z)|)$ и возьмем
$M:=(1-\varepsilon)\overline Q$. Тогда $M\in CC(Q)$ и
$$
h(Q,b,r)(b)-h(M,b,r)(b)\le \max\limits_{|z|\le r+1}|H_Q(z)-H_M(z)|<\frac{\alpha}{2}.
$$
Поэтому для любого $b\in R_Q$
$$
h(M,b,r)(b)-H_Q(b)\ge h(Q,b,r)(b)-\alpha/2-H_Q(b)\ge\alpha/2.
$$

\end{proof}

%%%%%%%%%%%%%%%%%%%%%%%%%%%%%%%%%%%%%%%%%%%%%%%%%%%%%%%%%%%%%%%%%%%%%%%%%%%%%%%%%%%%%%%%%%%%

Приведем определение, идущее от Л.~Эренпрайса \cite{EHRDIV}
и К.А.~Беренстейна, Б.А.~Тейлора \cite{BERTAY}.

\begin{definition}
Пусть $\mu\in A(K)'$, $A$ -- замкнутое подмножество $S$.
Функция $\widehat\mu$ называется {\it медленно убывающей} на $\Gamma(A)$,
если
$\forall k\in \mathbb{N}\,
\, \exists R > 0:
\forall z \in \Gamma(A), \,
|z| \geqslant R, \,
\exists w\in\mathbb{C},$
$|w - z|\leqslant |z|/k:\,
|\widehat \mu (w)| \geqslant \exp(H_K(w) - |w|/k).
$
\end{definition}

Отметим, что медленное убывание $\widehat\mu$ на $ \Gamma (A)$ равносильно тому, что
индикатор $h_{\widehat{\mu}}(\theta)$ целой функции $\widehat\mu$ экспоненциального типа
равен $H_K(e^{i\theta})$, $\theta\in\mathbb R$, и $\widehat\mu$ -- функция вполне регулярного роста на
$\Gamma (A)$ в смысле Левина-Пфлюгера \cite{LEVIN}.

Далее $V(\widehat{\mu})$ -- множество всех
нулей функции $\widehat\mu$. Если  $V(\widehat{\mu})$ бесконечно, то
символ $A_{\widehat \mu}$ обозначает множество всех предельных точек
последовательности
$\left(z/ |z| \,| \, z\in V(\widehat{\mu})\right)$.
Для множества $B\subset\mathbb C$ введем конус, порожденный $B$:\,
$\Gamma(B):=\{\lambda b\,|\, \lambda \geqslant 0,
b\in B \}$.

Согласно \cite[лемма 5; см. также
утверждение 4 там же]{MOMM94}
медленное убывание $\widehat\mu$ на $\Gamma(A_{\widehat\mu})$
гарантирует "хорошие" \, оценки снизу для $|\widehat\mu|$ на границе некоторой
окрестности нулевого множества $\widehat\mu$, состоящей из
"небольших" \, компонент. Именно, справедлива

\medskip
\begin{lemma}\label{muLemma2} \cite{MOMM94}
Пусть множество всех нулей $\widehat\mu$ бесконечно и $\widehat{\mu}$ медленно убывает на $\Gamma(A_{\widehat{\mu}})$.
Существует последовательность попарно непересекающихся ограниченных областей
$(\Omega_j)_{j\in\NN}$ таких, что любого $n\in\NN$  найдется постоянная $C>0$,
для которой для любого $j\in\NN$ выполняется следующее:

\begin{itemize}
\item[(i)]
$|\widehat \mu (z)| \geqslant {\rm exp}
 (H_K(z) - |z|/n - C), \,\, z\in\partial \Omega_j$.
\item[(ii)]
$\sup\limits_{z,w\in \Omega_j}|z-w|\le \frac{1}{n} \inf\limits_{z\in \Omega_j}|z|+C$.
\item[(iii)]
$\sup\limits_{z\in \Omega_j}H_K(z)\le \inf\limits_{z\in \Omega_j}(H_K(z)+|z|/n)+C$.
\end{itemize}

Кроме того,
\begin{itemize}
\item[(iv)] $V(\widehat\mu)\subseteq \bigcup\limits_{j\in\NN} \Omega_j$ и
$\Omega_j \cap V(\widehat\mu) \neq \varnothing$
для любого $j\in \NN$.
\end{itemize}
\end{lemma}

\begin{lemma} \label{muLemma3}
Пусть оператор
$T_{\mu}: A(Q + K) \to A(Q)$ сюрьективен.
Тогда для любого компакта $A$ в $R_Q$ функция
$\widehat\mu$ медленно убывает на $\Gamma(A)$.
%для любого компакта $A\subset S_Q$.
\end{lemma}

\begin{proof} Далее пространства $A(Q)_b'$ и $A(Q+K)_b'$
отождествляются, по лемме 2.3, с $A_Q$ и $A_{Q+K}$, соответственно.
Поскольку оператор $T_{\mu} : A(Q + K) \to A(Q)$ сюрьективен, то
по теореме об открытом отображении \cite[теорема 6.7.2]{Edwards}
$T_\mu$ открыто. По \cite[теорема 8.6.8]{Edwards} для любого
равностепенно непрерывного множества $B\subset A_{Q + K}$ его прообраз
$M_{\hat\mu}^{-1} (B)$ -- равностепенно непрерывное подмножество $A_Q$.
Заметим, что для любого $M\in CC(Q)$, вследствие рефлексивности
пространства $A(M)$,
наборы ограниченных и равностепенно непрерывных подмножеств
$A_M$ (если $A(M)_b'$ отождествить с $A_M$) одни и те же.
Поэтому, учитывая замечание 2\,(iv), получим, что справедливо следующее
утверждение (BD):

{\it Пусть $(f_j)_{j\in\NN}$ -- последовательность в $A_Q$ такая, что
существует $M\in CC(Q)$, для которого для любого $s\in\NN$
$$
\sup\limits_{j\in\NN}
\sup\limits_{z\in\CC}
\frac{|\widehat\mu (z)||f_j (z)|}{{\rm exp}(H_M(z) + H_K (z) + |z| /s)} < +\infty.
$$
Тогда найдется компакт $T\in CC(Q)$ такой, что
для любого $s\in\NN$
$$
\sup\limits_{j\in\NN}
\sup\limits_{z\in\CC}
\frac{|f_j(z)|}{{\rm exp}(H_T (z)  + |z| /s)} < +\infty.
$$
}

(I) Пусть $Q$ ограниченно. Тогда $R_Q$ компактно. Докажем
утверждение для $A:= R_Q$.
Предположим, что %существует компакт $A\subset S_Q$, для которого
$\hat{\mu}$ не является медленно убывающей на $\Gamma(R_Q)$.
Тогда найдутся $k\in\NN$, последовательность $z_j \in \Gamma (R_Q),\, j\in\NN$,
для которых $|z_j| \to\infty$, $|z_j| > k$ и для любого $j\in\NN$
\begin{equation}
|\hat\mu (w) | < {\rm exp} (H_K(w) - |w|/k), \,\,\mbox{ если } \, |w - z_j| \leqslant
\frac{|z_j|}{k}.
\end{equation}
%%%%%%%%%%%%%%%%%%%%%%%%%%%%%%%%%%%%%%%%%%%%%%%%%%%%%%%%%%%%%%%%%%%%%%%%%%%%%%

В силу замечания 6~(II) найдется $k_1\geqslant 2k$ такое, что для любого компакта
$M\in CC(Q)$
\begin{equation}\label{delta}
\Delta \left(M, z_j, \frac{|z_j|}{k_1}\right) =
|z_j| \Delta \left(M, a_j, \frac{1}{k_1}\right) \le \frac{|z_j|}{2 k}, \,
j\in\NN,
\end{equation}
где $a_j := z_j / |z_j|$.

По лемме 3.2\,(ii) существуют компакт
$M\in CC(Q)$ и $k_2\in\NN$, для которых
\begin{equation}
h\left(M, a_j, \frac{1}{k_1}\right)(a_j) - H_Q (a_j) \ge \frac{1}{k_2},\,\,j\in\NN.
\end{equation}

Положим $h_j (z) := h(M, z_j, |z_j|/{k_1})(z)$.
По \cite[теорема 4.4.4]{HERM} найдется $C_1 > 0$ такое, что
для любого $j\in\NN$ существует функция $f_j\in A(\CC)$, для которой
$f_j (z_j) := {\rm exp}(h_j (z_j))$ и
$$
|f_j (z)|\le C_1 {\rm exp} (\widetilde h_j (z) + C_1{\rm  log}(1 + |z|)),\, z\in\CC.
$$
Здесь $\widetilde h_j (z) := \sup\limits_{|t - z|\le 1} h_j (t), \, z\in\CC$,
\, $j\in\NN$.
Отсюда следует, что найдется постоянная $C_2$ такая, что для любых $z\in\CC$, $j\in\mathbb N$
\begin{equation}\label{fj}
|f_j(z)|\le C_2{\rm exp} \left(H_M (z) + C_1{\rm log}(1 + |z|) + \Delta \left(M, z_j,
\,\frac{|z_j|}{k_1}\right)\right).
\end{equation}
Поэтому, в частности, все функции $f_j$ принадлежат $A_Q$.

В силу (3.2), (3.3), (3.5) и того, что $k_1 \ge 2k$, для $j\in\NN$
$$
|\widehat\mu (z) f_j (z)| \le C_2 {\rm exp} (H_M (z) + C_1
{\rm log} (1 + |z|) + H_K (z)),
\,\mbox{ если } \, |z - z_j| \le \frac{|z_j|}{k_1}.
$$
Поскольку $\mu\in A(K)'$, то \cite[теорема 4.5.3]{HERM}
для любого $\varepsilon>0$
$$
\sup\limits_{z\in\CC}\frac{|\widehat\mu(z)|}{\exp(H_K(z)+\varepsilon|z|)}<+\infty.
$$
Таким образом, для любого $s\in\NN$
$$
\sup\limits_{j\in\NN} \sup\limits_{z\in\CC}
\frac{|\widehat\mu (z) f_j (z)|}{{\rm exp} (H_M (z) + H_K (z) + |z|/s)} < +\infty.
$$
При этом для любого $j\in\NN$, вследствие (3.4),
$$
f_j (z_j) = {\rm exp} \left(|z_j| h\left(M, a_j,
\frac{1}{k_1}\right)(a_j)\right)\ge {\rm exp} \left(H_Q (z_j) +
\frac{|z_j|}{k_2}\right),
$$
а значит, для $s_0:= 2k_2$
$$
\sup\limits_{z\in\CC} \frac{|f_j (z)|}{{\rm exp} (H_Q (z) +
|z| / s_0)} \ge {\rm exp} \left(\frac{|z_j|}{2k_2}\right), \,\,
j\in\mathbb N.
$$
Отсюда следует, что для любого $T\in CC(Q)$
$$
\sup\limits_{z\in\CC} \frac{|f_j (z)|}{{\rm exp} (H_T (z) +
|z| / s_0)} \ge {\rm exp} \left(\frac{|z_j|}{2k_2}\right), \,\,
j\in\mathbb N.
$$
Это противоречит утверждению (BD). Таким образом, $\widehat\mu$
медленно убывает на $\Gamma(R_Q)$.

(II) Пусть множество $Q$ неограниченно и $A$ -- компакт
в $R_Q$. Без ограничения общности $0\in{\rm int}\,Q$ и $0\in K$.
Зафиксируем разделяющую $Q$ прямую $l$ (см. замечание 5).
Из леммы 2.1 вытекает, что $Q(l)$ имеет счетный базис
окрестностей, состоящий из (ограниченных) выпуклых областей. Вследствие
замечания 4~(ii)
существует компактная окрестность $V$ множества $S_\omega$ в
$S$ такая, что $H_{Q(l)}=H_Q$ на $V$.
Если $\widehat{\mu}$ не является медленно убывающей на
$\Gamma(V\cap A)$, то по части (I) существуют компакт $M\in CC(Q(l))
\subset CC(Q)$,
точки $z_j\in \Gamma(V\cap A)$, $j\in\NN$, последовательность
функций $(f_j)_{j\in\mathbb N}$ из $H_{Q(l)}$, для которых
для любого $s\in\NN$
$$
\sup\limits_{j\in\NN}\sup\limits_{z\in\CC}
\frac{|\widehat{\mu}(z)||f_j(z)|}{\exp(H_M(z)+H_K(z)+|z|/s)} < +\infty,
$$
но найдется $s_0\in\NN$ такое, что
$$
\sup\limits_{j\in\NN}
\frac{|f_j(z_j)|}{\exp(H_{Q(l)}(z_j)+|z_j|/s_0)} =  +\infty.
$$
Поскольку $H_{Q(l)}(z_j)=H_Q(z_j)$, \, $j\in\NN$,
то для любого $T\in CC(Q)$
$$
\sup\limits_{j\in\NN}\sup\limits_{z\in\CC}
\frac{|f_j(z)|}{\exp(H_T(z)+|z|/s_0)} = +\infty,
$$
Значит, утверждение (BD) не выполняется.

Пусть $B:=S\backslash({\rm int}_r V)$,
где ${\rm int}_r V$ -- относительная внутренность $V$ в $S$.
Предположим, что $\widehat{\mu}$ не является медленно
убывающей на $A\cap B$. Тогда, согласно \cite[теорема 1]{EPIF74},
%\cite[теорема 1; \S~3, доказательство теоремы 1]{TKACH77},
\cite[теорема 1; раздел 3, доказательство теоремы 1]{TKACH77},
найдутся
выпуклый компакт $M_0$ в ${\rm int}\,Q$, содержащий $0$,
целая функция
$f$, последовательность $z_j\in\Gamma(A\cap B)$, $j\in\NN$,
такие, что для любого $s\in\NN$
$$
\sup\limits_{z\in\CC}
\frac{|\widehat{\mu}(z)||f(z)|}{\exp(H_{M_0}(z)+H_K(z)+|z|/s)} < +\infty,
$$
а для любого $T_0\in CC({\rm int}\,Q)$ существует $\widetilde s\in\NN$,
для которого
$$
\sup\limits_{j\in\NN}
\frac{|f(z_j)|}{\exp(H_{T_0}(z_j)+|z_j|/{\widetilde s})} = +\infty.
$$
При этом $z_j$ можно выбрать так, что $\widehat{\mu}(z_j)\ne 0$
для любого $j\in\NN$.
По
\cite[теорема 4.4]{KRTER}
 найдется последовательность
многочленов $(p_n)_{n\in\NN}$ такая, что последовательность
$(\widehat{\mu}p_n)_{n\in\NN}$
сходится к $\widehat{\mu}f$ в $A_{M_0+K}$ и, как следствие,
в каждой точке $z_j$. Отсюда и из замечания 4~(iii) следует, что
для некоторой подпоследовательности
$(p_{n_j})_{j\in\NN}$ не выполняется условие (BD).
Таким образом, $\widehat\mu$ медленно убывает на $\Gamma(A\cap B)$
и на $\Gamma(A)\subset\Gamma(A\cap B)\cup\Gamma(A\cap V)$.

\end{proof}

%%%%%%%%%%%%%%%%%%%%%%%%%%%%%%%%%%%%%%%%%%%%%%%%%%%%%%%%%%%%%%%%%%%%%%%%%%%%%%%%%%%%%%%%%%%%%%%%%%%%%%%%%%%%%%%%%%%%%%

%%%%%%%%%%%%%%%%%%%%%%%%%%%%%%%%%%%%%%%%%%%%%%%%%%%%%%%%%%%%%%%%%%%%%%%%%%%%%%%%%%%%%%%%%%%%%%%%%%%%%%%%%%%%%%%%%%%%%%

\subsection{Вспомогательные пространства векторнозначных
последовательностей}

Пусть множество нулей $\widehat\mu$ бесконечно и некоторая
последовательность попарно
непересекающихся ограниченных областей $(\Omega_j)_{j\in\NN}$
удовлетворяет условиям (i) -- (iv) леммы \ref{muLemma2}.
Следующая конструкция идет от Р.~Майзе \cite{MEISE85}. Через $A^{\infty}
(\Omega_j)$ обозначим банахово пространство всех ограниченных
аналитических в $\Omega_j$ функций; $I_j := {\widehat\mu}\,|_{\Omega_j}\cdot
A^{\infty}(\Omega_j)$ -- замкнутый идеал в $A^{\infty} (\Omega_j)$,
порожденный функцией $\widehat\mu |_{\Omega_j}$. Положим $E_j :=
A^{\infty}(\Omega_j)\big/ I_j$, $j\in\NN$; в $E_j$ вводится фактор-норма
$$
\|f + I_j\|_j := \inf\limits_{\xi \in f + I_j} \sup\limits_{z\in \Omega_j} |\xi (z)|.
$$
Размерность конечномерного пространства $E_j$ равна числу нулей функции $\widehat\mu$,
содерждащихся в $\Omega_j$, с учетом их кратностей.

Введем далее весовые пространства последовательностей с элементами из $E_j$.
Для любого $j\in\NN$ выберем $\nu_j\in \Omega_j\cap V(\widehat\mu)$. Без ограничения общности
можно считать, что $|\nu_j|\le|\nu_{j+1}|$ для любого $j\in\NN$. Так как $\widehat\mu$ --
целая функция экспоненциального типа, то $\lim\limits_{j\to\infty}\frac{{\rm log}\,j}{|\nu_j|}=0$.
Положим для $n,m\in\NN$
$$
\Lambda_m(Q_n,\widehat\mu,\mathbb E)
:=
\left\{X=(x_j)_{j\in\NN}\in\prod\limits_{j\in\mathbb N}E_j\, |\,\,
\,\||X\||_{nm}
:=\sup\limits_{j\in\NN}\frac{\|x_j\|_j}{\exp(H_{nm}(\nu_{j}))}
<+\infty\right\};
$$
$\Lambda_m(Q_n,\hat\mu,\mathbb E)$ -- банахово пространство с
нормой $\||X\||_{nm}$.
Пусть
$$
\Lambda(Q_n,\hat\mu,\mathbb E):={\rm ind}_{m\to}\Lambda_m(Q_n,\hat\mu,\mathbb E),\,\,
\Lambda(Q,\hat\mu,\mathbb E):={\rm proj}_{\gets n}\Lambda(Q_n,\hat\mu,\mathbb E).
$$
Отметим, что пространства $\Lambda(Q,\hat\mu,\mathbb E)$ от выбора точек
$\nu_j\in\Omega_j\cap V(\widehat\mu)$, $j\in\mathbb N$, не зависят.

%%%%%%%%%%%%%%%%%%%%%%%%%%%%%%%%%%%%%%%%%%%%%%%%%%%%%%%%%%%%%%%%%%%%%%

\medskip
Доказательства следующих ниже лемм 3.5 и 3.6 (для произвольного выпуклого компакта
$K\subset\CC$), по сути, такие же, как и доказательства
лемм 3.1 и 3.3 в \cite{BM} для $K = \{0\}$. Доказательство леммы 3.5 сводится к проверке
некоторого условия (P), равносильного ультраборнологичности пространства числовых последовательностей
$\Lambda(B)$, задаваемого числовой матрицей $B$ (при этом используется
\cite[теорема 4.2]{VOGT}). Эта матрица $B$ получается из той, которая задает
$\Lambda (Q, \hat\mu, \mathbb E)$, повторением столбцов последней (число повторений равно
размерности соответствующих пространств $E_j$). Топологическая изоморфность
$\Lambda (Q, \widehat\mu, \mathbb E)$ и $\Lambda(B)$ вытекает (по доказательству
\cite[предложение 1.4]{MEISE85}) из того, что
$\lim\limits_{j\to\infty}\frac{{\rm log}\,j}{|\nu_j|}=0$, и "малости" \, диаметров компонент
$\Omega_j$ по сравнению с их расстоянием до начала координат, т.~е. из условия (iii) леммы 3.3.

\begin{lemma}
Пусть множество нулей функции $\widehat\mu$ бесконечно
и существует
последовательность областей $(\Omega_j)_{j\in\NN}$, как в лемме 3.3.
Пространство $\Lambda (Q, \widehat\mu, \mathbb E)$
ультраборнологично.
\end{lemma}

%%%%%%%%%%%%%%%%%%%%%%%%%%%%%%%%%%%%%%%%%%%%%%%%%%%%%%%%%%%%%%%%%%%%%%%%%%%%%%%%%%%%%%%%%

\begin{lemma}
Пусть множество нулей функции $\widehat\mu$ бесконечно;
оператор $T_{\mu}: A(Q + K) \to A(Q)$ сюрьективен и существует
последовательность областей $(\Omega_j)_{j\in\NN}$, как в лемме 3.3.
Тогда отображение
$$
\rho : A_{Q + K}\left/\right.(\widehat\mu \cdot A_Q) \to \Lambda
(Q, \widehat\mu, \mathbb E), \,\, f + \widehat\mu \cdot A_Q
\mapsto \left(f\left|\right._{S_j}+I_j\right)_{j\in\mathbb N}
$$
является линейным топологическим изоморфизмом "на".
\end{lemma}

\begin{proof} 
Как установлено в
\cite[доказательство предложения 6]{MOMM94}, для любого $n\in\NN$ точна короткая последовательность
$$
0\longrightarrow \mathcal A_{Q_n}\stackrel{M_{\widehat\mu}}{\longrightarrow}
\mathcal A_{Q_n+K}\stackrel{\rho_{0,n}}{\longrightarrow}
\Lambda(Q_n+K,\widehat\mu,\mathbb E)\longrightarrow 0,
$$
для $\rho_{0,n}(f):=(f|_{\Omega_j}+E_j)_{j\in\NN}$.
Поскольку ${\rm Proj}^1\mathcal A_{Q,{\rm sp}}=0$ (замечание 2), то
 \cite[теорема 5.1]{VOGT}, отображение
$\rho_0: \mathcal A_{Q+K}\to \Lambda(Q+K,\widehat\mu,\mathbb E)$,
$\rho_{0,n}(f):=(f|_{\Omega_j}+E_j)_{j\in\NN}$, сюръективно.
Равенство $\widehat{\mu}\cdot  \mathcal A_Q =(\widehat{\mu}\cdot A(\CC))\cap \mathcal A_{Q+K}$
(лемма 3.1) влечет, что ядро оператора
$\rho_0: \mathcal A_{Q+K}\to \Lambda(Q+K,\widehat\mu,\mathbb E)$
совпадает с $\widehat\mu \cdot \mathcal A_Q$.
Поскольку пространство $\mathcal A_Q$, как счетный проективный предел (LB)-пространств,
имеет сеть \cite[лемма 24.28]{MEIVOGT}, а $\Lambda (Q, \widehat\mu,\mathbb E)$
ультраборнологично (лемма 3.5), то по теореме об открытом оторажении \cite[теорема 24.30]{MEIVOGT}
$\rho : \mathcal A_{Q + K}\bigm / (\widehat\mu \cdot \mathcal A_Q) \to \Lambda (Q, \widehat\mu,\mathbb E)$ --
топологический изоморфизм "на".
\end{proof}

Полученная в лемме 3.6 реализация факторпространства
$A_{Q+K}\!\!\bigm/\!\!(\hat{\mu}\cdot \mathcal A_Q)$ в виде пространства
векторнозначных последовательностей является важным промежуточным звеном между
общими и выражаемыми посредством субгармонических функций
условиями наличия ЛНПО к $T_\mu$.

%%%%%%%%%%%%%%%%%%%%%%%%%%%%%%%%%%%%%%%%%%%%%%%%%%%%%%%%%%%%%%%%%

\subsection{Критерий наличия ЛНПО к оператору свертки в терминах
существования специальных семейств субгармонических функций}

Введем следующее условие существования специальных
семейств субгармонических функций.

\begin{definition}
Пусть $A\subset S=\{z\in\mathbb C\,|\,|z|=1\}$.
Будем говорить, что выполняется условие $SH(A,Q)$,
если существует семейство
субгармонических
в $\CC$ функций $u_t$, $t\in\Gamma(A)$, таких, что
$u_t(t)\geqslant 0$, $t\in\Gamma(A)$, и
$\forall n \,\exists n'\, \forall m \,\exists m':$
$$
u_t(z)\leqslant H_{nm'} (z) - H_{n'm} (t), \,\, z\in\CC, \, t\in\Gamma(A).
$$
\end{definition}

\medskip
Для ограниченного множества $Q$
условие $SH(A,Q)$ имеет более простой вид (см. \S~4, замечание 7).

Вследствие \cite[доказательство замечания 4.3]{BM}
справедлива

%\medskip
\begin{lemma}\label{SHLemma2}
Пусть для множества $A\subset S$ выполняется условие $SH(A, Q)$.
Тогда $A\subset R_Q$.
\end{lemma}

%\medskip
\begin{theorem} Пусть множество нулей
$\widehat\mu$ бесконечно и
оператор $T_{\mu} : A(Q + K)\to A(Q)$ сюрьективен.
Следующие утверждения равносильны:

\noindent
(i) $T_{\mu} : A(Q + K)\to A(Q)$ имеет ЛНПО.

\noindent
(ii) Выполняется условие $SH(A_{\widehat \mu}, Q)$.
\end{theorem}

\begin{proof}
$(i)\Rightarrow(ii)$: Используем идею представления оператора $T_{\mu}$
  в виде произведения (композиции) операторов свертки,
  идущую от Ю.Ф. Коробейника
%  \cite[\S~5]{SMJ}.
  \cite{SMJ}.

  Существует целая функция $g$ типа 0 при порядке 1
  такая, что
  $h := \widehat\mu/g \in A(\CC)$ и
  $A_{\widehat\mu} = A_g$.
  Тогда $T_{\mu} = T_{{\cal F}^{-1} (g)}T_{{\cal F}^{-1}(h)}$.
  Так как $T_{\mu} : A(Q + K)\to A(Q)$ имеет ЛНПО, то оператор
  $T_{{\cal F}^{-1}(g)} : A(Q)\to A(Q)$ также имеет ЛНПО.
  Поэтому утверждение этой леммы вытекает из
  сответствующего утверждения для оператора $T_{{\cal F}^{-1} (g)}$
  \cite[теорема 4.2, замечание 4.3]{BM}.

\noindent
  $(ii)\Rightarrow (i)$: Пусть выполняется условие
  $SH(A_{\widehat \mu}, Q)$. По лемме 3.7
  $A_{\widehat\mu}\subset R_Q$. Вследствие леммы 3.4
функция $\hat\mu$ медленно убывает на
  $\Gamma (A_{\widehat \mu})$.
  Тогда существует последовательность областей
  $(\Omega_j)_{j\in\mathbb N}$, как в лемме 3.3. В силу леммы 3.6
   отображение
   $$
   \rho: A_{Q+K}\!\!\bigm/\!\!(\hat{\mu}\cdot A_Q) \to \Lambda(Q,\widehat{\mu},\mathbb E), \,\,
   f+\hat{\mu}\cdot A_Q \mapsto (f\!\!\bigm|_{S_j} + I_j)_{j\in\NN}
   $$
   является топологическим изоморфизмом "на"\,.
Поступая, по сути, так же, как при доказательстве импликации
$(ii)\Rightarrow(i)$ в \cite[теорема 4.2]{BM} получим, что
фактор-отображение $q: A_{Q + K} \to A_{Q + K} / (\hat\mu \cdot A_Q)$
имеет ЛНПО. (Переход от "исключительного" \, множества
кругов $(B_j)_{j\in\mathbb N}$ к "исключительному" \,
множеству областей $(\Omega_j)_{j\in\mathbb N}$
не приводит к затруднениям.)
По лемме 2.5 оператор $T_{\mu} : A(Q + K)\to A(Q)$
имеет ЛНПО.
\end{proof}

\medskip

\begin{flushleft}
С.~Н.~Мелихов

Южный федеральный университет, Ростов-на-Дону;
Южный математический институт Владикавказского научного центра РАН,
Владикавказ

E-mail: melih@math.rsu.ru

\end{flushleft}

\begin{flushleft}
Л.~В.~Ханина

Южный федеральный университет, Ростов-на-Дону

E-mail: khanina.lv@mail.ru

\end{flushleft}

\end{document}